\documentclass[fleqn,11pt]{article}
\usepackage{epic}
\usepackage{epsfig}
\usepackage[english]{babel}
\usepackage[mathscr]{eucal}

\usepackage{amssymb,epsf,amsmath}

\newtheorem{Theorem}{Theorem}[section]
\newtheorem{Proposition}[Theorem]{Proposition}

\newtheorem{Corollary}[Theorem]{Corollary}

\newtheorem{Ex}{Example}[section]

\newcommand{\Halmos}{\hfill$\Box$}
\newtheorem{Rem}{Remark}[section]

\newtheorem{proofhead}{Proof.}

\newenvironment{proof}
    {
        \par
        \begin{proofhead}
        \normalfont
    }
    {   \qed
        \end{proofhead}
        \par
    }

\newcommand{\bb}{\mathbb}

\def \R {{\bb R}}

\def \E{{\bb E}}

\def \var{{\bb V}{\rm ar}}

\def \cov{{\bb C}{\rm ov}}

\def \pr{{\bb P}}

\def \calT {\mathcal{T}}

\def \calV {\mathcal{V}}
\def \calW {\mathcal{W}}
\def \calK {\mathcal{K}}
\def \calM {\mathcal{M}}
\def \calN {\mathcal{N}}
\def \calH {\mathcal{H}}

\newcommand{\qed}{\hfill$\Box$}

\newcommand{\refs}[1]{(\ref{#1})}

\newcounter{mylistcnt}
\renewcommand{\themylistcnt}{{\rm({\roman{mylistcnt}})}}

\newcounter{zad}

\newcommand{\toi}{\to\infty}

\begin{document}
\title{On the asymptotics of supremum distribution for some iterated processes}

\author{
Marek Arendarczyk
\vspace*{.08in} \\
Mathematical Institute, University of Wroc\l aw \\
pl. Grunwaldzki 2/4, 50-384 Wroc\l aw, Poland }

\date{}
\maketitle

\begin{abstract}
In this paper, we study the asymptotic behavior of supremum distribution of some classes of
iterated stochastic processes $\{X(Y(t)) : t \in [0, \infty)\}$, where 
$\{X(t) : t \in \R \}$ is a centered Gaussian process and $\{Y(t): t \in [0, \infty)\}$ 
is an independent of $\{X(t)\}$ stochastic process with a.s. continuous sample paths. 
In particular, the asymptotic behavior of $\pr\left(\sup_{s \in [0,T]} X(Y(s)) > u\right)$ as $u \toi$, where $T > 0$,
as well as $\lim_{u\toi} \pr\left(\sup_{s \in [0, h(u)]} X(Y(s)) > u\right)$, for some suitably chosen function $h(u)$
are analyzed. 
As an illustration, we study the asymptotic behavior of the supremum distribution of iterated fractional Brownian motion process.

\end{abstract}

\section{Introduction}
Let $\{X(t): t \in \R \}$ and $\{Y(t): t \in [0, \infty) \}$ be two independent stochastic processes.
This contribution is devoted to the analysis of asymptotic behavior of supremum distribution
of iterated process $\{X(Y(t)) : t \in [0,\infty)\}$.

Originated by Burdzy \cite{Bur93, Bur94} for the case of iterated Brownian motion, the problem of analyzing the properties of iterated processes was intensively studied in recent years. 
Motivation for the analysis of the process $\{X(Y(t))\}$ in case of $\{X(t)\}$ and $\{Y(t)\}$ being independent Brownian motions was
delivered by its connections to the $4$th order PDE's (see, e.g., \cite{Funaki79, All01, Nour08}).
A vast literature is devoted to the analysis of many interesting probabilistic properties
of iterated Brownian motions (see, e.g., \cite{Bur95, Hu95, Shi95, Bertoin96, Khosh96, Eis99, Khosh99}).
We also refer to \cite{Curien11} where convergence of finite dimensional distributions
of $n$th iterated Brownian motion is studied and \cite{Turban04} 
where infinite iterations of i.i.d. random walks are analyzed.  

Recent studies also focus on properties of $\{X(Y(t)): t \in [0, \infty)\}$
for the case of more general Gaussian processes $\{X(t)\}$.
One of interesting example of such processes is fractional Laplace motion $\{B_H(\Gamma(t)): t\in [0,\infty) \}$,
where $\{\Gamma(t): t \in [0,\infty)\}$ is a Gamma process.
Motivation for analyzing fractional Laplace motions stems from hydrodynamic
models (see, e.g., \cite{Koz04}). This kind of processes were described in \cite{Koz06}, see
also \cite{ArD11} where asymptotic behavior of exit-time distribution for the process $\{B_H(\Gamma(t))\}$ was found.
Another important class of iterated processes are the so-called $\alpha$-time fractional Brownian motions $\{B_H(Y(t))\}$, where $\{Y(t)\}$ is $\alpha$-stable subordinator independent of the process $\{B_H(t)\}$ (see, e.g., \cite{Linde04, Nane06, Linde08, Lif09}). 
We also refer to \cite{Michna98} and \cite{Deb14} where the process $\{B_H(Y(t))\}$
was analyzed in the context of theoretical actuarial models. 

The process $\{B_H(Y(t))\}$ in the case of $\{Y(t)\}$ not being a subordinator was studied in \cite{Lif09}.
In this case, the small deviations asymptotics was found for the so-called iterated fractional Brownian motion process
$\{B_{H_2}(B_{H_1}(t))\}$, where $\{B_{H_1}(t)\}$, $\{B_{H_2}(t)\}$ are independent fractional Brownian motions 
with Hurst parameters $H_1, H_2 \in (0,1]$ respectively.  

In this paper, we focus on the analysis of asymptotic behavior of supremum distribution of 
the process $\{X(Y(t)): t \in [0, \infty)\}$ for general classes of stochastic processes $\{X(t)\}$, $\{Y(t)\}$ with a.s. continuous 
sample paths. 

{\it Notation and organization of the paper:}

In Section \ref{sec.short}, we study the asymptotic behavior of
\begin{equation}		\label{main.short}
		\pr\left(\sup_{s \in [0,T]} X(Y(s)) >u \right) \ \ {\rm as} \  u \toi,
\end{equation}
where $T > 0$ and $\{X(t): t \in \R \}$, $\{Y(t): t \in [0,\infty)\}$ are independent stochastic processes.
This problem is closely related to the analysis of asymptotic behavior of the supremum 
distribution of the process $\{X(t)\}$ over a random time interval (see, e.g., \cite{BDZ04, ArD11, ArD12, Tan13, Deb14}). 

We start in Section \ref{sec.main.short} by giving general result for the case of $\{X(t)\}$ being 
Gaussian process with stationary increments and convex variance function (see Section \ref{sec.main.short}, assumptions {\bf A1 -- A3}).
In this case, under some general conditions on the process $\{Y(t)\}$ 
(see Section \ref{sec.main.short}, assumptions {\bf L1, L2}), we show that
\refs{main.short} reduces to 
\begin{equation}	\label{sup.random}
		\pr\left(\sup_{s \in [0,\calT]} X(s) > u\right) \ \ {\rm as} \  u \toi,
\end{equation}
where $\calT$ is a non-negative random variable independent of $\{X(t)\}$ with asymptotically Weibullian
tail distribution, that is,
\begin{equation}		\label{eq.def.Weibull}	
		\pr\left(\calT > u \right)
	=	
		C u^\gamma \exp(- \beta u^\alpha)(1 + o(1))
\end{equation}
as $u \toi,$ where $\alpha, \beta, C > 0, \gamma \in \mathbb{R}$ (see, e.g., \cite{ArD11} for details). 
We write 
$
		\calT \in \mathcal{W}(\alpha, \beta, \gamma, C)
$
if $\calT$ satisfies \refs{eq.def.Weibull}.

Section \ref{sec.fbm} is devoted to the special case of the process $\{B_H(Y(t)) : t \in [0, \infty)\}$,
where $\{B_H(t) : t \in \R\}$ is a fractional Brownian motion ({\it fBm}) with Hurst parameter $H \in (0, 1]$,
that is, a centered Gaussian process with stationary increments, a.s. continuous sample paths, $B_H(0) = 0$, and covariance function  
$
		\cov(B_H(t), B_H(s)) = \frac{1}{2}\left(|s|^{2H} + |t|^{2H} - |t - s|^{2H}\right).
$
Due to self-similarity of the process $\{B_H(t)\}$, we are able to provide the exact
asymptotics of \refs{main.short} for the whole range of Hurst parameters $H \in (0, 1]$.
As an illustration, in Proposition \ref{prop.iter.fbm}, we work
out the exact asymptotics of the supremum distribution of iterated fractional Brownian motion
$\{B_{H_2}(B_{H_1}(t)) : t \in [0,\infty)\}$, where $\{B_{H_1}(t)\}, \{B_{H_2}(t)\}$ are
independent fractional Brownian motions with Hurst parameters $H_1, H_2$ respectively.
Note that small deviation counterpart of this problem was recently studied in \cite{Lif09}.

In Section \ref{sec.short.st}, the case of $\{X(t)\}$ being a stationary Gaussian process is analyzed (see Section \ref{sec.short.st}, assumptions {\bf D1, D2}). 
In this case the exact asymptotics of \refs{main.short} can be achieved under a 
general condition of finite average {\it span} of the process $\{Y(t)\}$ (see Section \ref{sec.short.st}, assumption {\bf S1}).
This problem is strongly related to the analysis of \refs{sup.random} in case of $\calT$ being a random variable 
with finite mean.
In this case the asymptotics of \refs{sup.random} has the form 
(see \cite{ArD12}, Theorem 3.1, and also \cite{Pic69} for the classical result of Pickands' on deterministic time interval)
\[
		\pr\left(\sup_{s \in [0,\calT]} X(s) > u\right)
	=	
		\E\calT C^{1/\alpha}\calH_\alpha u^{2/\alpha} \Psi(u)(1 + o(1))
\] 
as $u \toi$, where $\mathcal{H}_\alpha$ is the Pickands' constant defined by the limit
\[
		\mathcal{H}_\alpha
	=
		\lim_{T \toi}\frac{1}{T} \E\exp\left(\sup_{t \in [0,T]} \sqrt{2}B_\frac{\alpha}{2}(t) - t^{\alpha}\right),
\]
and 
$
	\Psi(u) := \pr(\calN > u)
$ 
with $\calN$ denoting the standard normal random variable.

In the second part of the paper, we study 
\begin{equation}	\label{main.long}
		\lim_{u \toi} \pr\left(\sup_{s \in [0,h(u)]} X(Y(s)) >u \right),
\end{equation}
for some suitably chosen function $h(u)$.\\ 
First, in Theorem \ref{long.st.incr} we investigate limiting behavior of \refs{main.long} for the case of $\{X(t)\}$ and $\{Y(t)\}$ being independent Gaussian processes with stationary increments that satisfy some general regularity conditions 
(see Section \ref{sec.long.st}, assumptions {\bf B1 -- B3}).
Then, in Theorem \ref{long.stat.main} and Proposition \ref{long.st.ss}, the case of $\{X(t)\}$ being stationary Gaussian process
is studied.
We analyze $\{X(Y(t))\}$ for both weakly and strongly dependent stationary Gaussian processes $\{X(t)\}$
(see Section \ref{sec.long.st}, assumptions {\bf D1 -- D3}). In these settings we provide \refs{main.long} 
in the case of $\{Y(t)\}$ 
being a centered Gaussian process with stationary increments,
as well as for self-similar process $\{Y(t)\}$ that is not necessarily Gaussian.
\section{Short timescale case}	\label{sec.short}
In this section, we study the asymptotic behavior of 
\begin{equation}		\label{main.short2}
		\pr\left(\sup_{s \in [0,T]} X(Y(s)) >u \right) \ \ {\rm as} \  u \toi,
\end{equation}
where $T > 0$, for the case of $\{X(t): t \in \R \}$ being a centered Gaussian process 
with a.s. continuous sample paths. We focus on two important classes of Gaussian processes.
First, processes $\{X(t)\}$ with stationary increments are studied. Then, we analyze the case of stationary processes $\{X(t)\}$. 
%
\subsection{The stationary increments case}	\label{sec.main.short}

Let $\{X(t): t \in \R \}$ be a centered Gaussian process with stationary increments,
a.s. continuous sample paths, $X(0) = 0$ a.s., and variance function 
$\sigma_X^2(t) := \var(X(t))$ that satisfies the following assumptions 
\\ \\
{\bf A1}\ \  $\sigma^2_X(t)\in C^1([0, \infty))$ is convex;\\
{\bf A2}\ \  $\sigma^2_X(t)$ is regularly varying at $\infty $ with parameter
$\alpha_\infty\in (1 , 2)$; \\
{\bf A3}\ \  there exists $D>0$ such that $\sigma^2_X(t)\le Dt^{\alpha_\infty}$
for each $t\ge 0$.\\ \\
To provide general result for \refs{main.short2} we assume that $\{Y(t) : t \in [0, \infty) \}$ is a stochastic process
with a.s. continuous sample paths, which is independent of $\{X(t)\}$ and its extremal distributions  
belong to the Weibullian class of random variables, that is,
\\ \\
%
%
{\bf L1}\ \  $\calM := \sup_{s \in [0,T]} Y(s) \in \calW(\alpha_1, \beta_1, \gamma_1, C_1)$, \ {\rm with} \ 
		$\alpha_1, \beta_1, C_1 > 0, \gamma_1 \in \R$;\\
{\bf L2}\ \  $\calK:= -\inf_{s \in [0,T]} Y(s) \in \calW(\alpha_2, \beta_2, \gamma_2, C_2)$, \ {\rm with} \ 
		$\alpha_2, \beta_2, C_2 > 0, \gamma_2 \in \R$. 
\begin{Rem}
Note that assumptions {\bf L1, L2} cover, e.g., a class of general Gaussian processes.
\end{Rem}
In the following theorem we present structural form of the asymptotics. The explicit asymptotic
expansion is presented in Corollary \ref{cor.exact}.
\begin{Theorem}		\label{th.short}
Let $\{ X(t): t \in \R \}$ be a centered Gaussian process with stationary increments 
and variance function $\sigma^2_X(t)$ that satisfies assumptions {\bf A1 -- A3}
and $\{Y(t): t \in [0,\infty)\}$ be an independent of 
$\{X(t)\}$ stochastic process with a.s. continuous sample paths that satisfies {\bf L1, L2}. If 
\begin{itemize}
\item[{\rm (i)}] $\pr(\calK > u) = o(\pr(\calM > u))$ as $u \toi$, then
\[
		\pr\left(\sup_{s \in [0,T]} X(Y(s)) >u\right)
	=
			\pr\left(X(\calM) > u\right)(1 + o(1)) \ \ \ {\rm as} \  u \toi;
\]
\item[{\rm (ii)}] $\pr(\calM > u) = o(\pr(\calK > u))$ as $u \toi$, then
\[
		\pr\left(\sup_{s \in [0,T]} X(Y(s)) >u\right)
	=	
		\pr\left(X(\calK) > u\right)(1 + o(1)) \ \ \ {\rm as} \ u \toi; 
\]
\item[{\rm (iii)}] $\pr(\calK > u) = \frac{C_2}{C_1}\pr(\calM > u)(1 + o(1))$, as $u \toi$, then
\[
			\pr\left(\sup_{s \in [0,T]} X(Y(s)) >u\right)
	=
			\left(\pr(X(\calM) > u) + \pr(X(\calK) > u)\right)(1 + o(1))\ \ \  {\rm as}\ u \toi.
\]
\end{itemize}
\end{Theorem}
The proof of Theorem \ref{th.short} is presented in Section \ref{th.short.proof}. \\

If the variance function of $\{X(t)\}$ is regular enough, then the straightforward application of Corollary 3.2 in \cite{ArD11} enables us to give the exact form of the asymptotics.

\begin{Corollary}\label{cor.exact}
Let $\{X(t): t \in \R\}$ be a centered Gaussian process with stationary increments
and variance function that satisfies {\bf A1} and $\{Y(t): t \in [0,\infty)\}$ be an independent of $\{X(t)\}$ 
stochastic process with a.s. continuous sample paths that satisfies {\bf L1, L2}.
Additionally, if $\sigma_X^2(t)=Dt^{\alpha_\infty}+o(t^{\alpha_\infty-\alpha})$, as $t\to\infty$,
with $\alpha_\infty\in(1,2)$, $D>0$, and $\alpha = \min(\alpha_1, \alpha_2)$, then
\begin{eqnarray}
    \nonumber
\sup_{s \in [0,T]} X(Y(s))\in \mathcal{W}(\widetilde{\alpha},\widetilde{ \beta},\widetilde{\gamma},\widetilde{ C}),
\end{eqnarray} 
where
\begin{eqnarray*}
\widetilde{\alpha}&=&\frac{2\alpha}{\alpha+\alpha_\infty},\ \ \ \ \ \ 
\widetilde{ \beta}\ = \ 
    \beta^{\frac{\alpha_\infty}{\alpha+\alpha_\infty}}\left(\frac{D}{2}\right)^{\frac{\alpha}{\alpha+\alpha_\infty}}
     \left( \left(\frac{\alpha}{\alpha_\infty} \right)^{\frac{\alpha_\infty}{\alpha+\alpha_\infty}}  +
             \left(\frac{\alpha_\infty}{\alpha} \right)^{\frac{\alpha}{\alpha+\alpha_\infty}} \right),		\\
\widetilde{\gamma}&=&\frac{2\gamma}{\alpha+\alpha_\infty},\ \ \ \ \ \ 
\widetilde{ C}\ = \ 
     CD^{-1/\alpha_\infty}\sqrt{\frac{\alpha_\infty}{2(\alpha+\alpha_\infty)}}\left( \frac{\alpha_\infty}{2\alpha\beta} D^{\alpha_\infty/\alpha} \right)^{\frac{\gamma}{\alpha+\alpha_\infty}},
\end{eqnarray*}
with
\begin{eqnarray*}
		&&(\beta, \gamma, C)
	=
		\left\{ \begin{array}{lll}
		(\beta_1, \gamma_1, C_1) & for & \pr(\calK > u) = o(\pr(\calM > u)) \ \ as\ u \toi,\\
		(\beta_2, \gamma_2, C_2) & for & \pr(\calM > u) = o(\pr(\calK > u)) \ \ as\ u \toi,\\
		(\beta_1, \gamma_1, C_1 + C_2) & for & \pr(\calK > u) = \frac{C_2}{C_1}\pr(\calM > u)(1 + o(1)) \ \ as\ u \toi.
		\end{array} \right.
\end{eqnarray*}

\end{Corollary}


\subsection{The case of fBm}	\label{sec.fbm}
Let $\{B_H(t): t \in \R \}$ be a fractional Brownian motion with Hurst parameter $H \in (0,1]$.
In this section, we analyze the asymptotic behavior of 
\begin{eqnarray}	\label{sup.short.fbm}
		\pr\left(\sup_{s \in [0,T]} B_H(Y(s)) >u\right)\ \ {\rm as} \  u \toi,
\end{eqnarray}
where $T > 0$ and $\{Y(t) : t \in [0, \infty)\}$ is an independent of $\{B_H(t)\}$ stochastic process
with a.s. continuous sample paths that satisfies assumptions
{\bf L1, L2}. Due to self-similarity of the process 
$\{B_H(t)\}$, we are able to provide the exact asymptotics of \refs{sup.short.fbm} for the whole range of
Hurst parameters $H \in (0, 1]$, which includes cases of both convex and concave variance functions. 
\begin{Proposition}		\label{th.fbm}
Let $\{ B_H(t): t \in \R \}$ be a fractional Brownian motion with Hurst parameter
$H \in (0,1]$ and $\{Y(t): t \in [0,\infty)\}$ be an independent of 
$\{B_H(t)\}$ stochastic process  with a.s. continuous sample paths that satisfies {\bf L1, L2}. If: 

\begin{itemize}
\item[] $H \in (0,1/2)$, then
\[
		\sup_{s \in [0,T]} B_H(Y(s)) \in \mathcal{W}\left(\frac{2\alpha}{\alpha + 2H}, \tilde{{\beta}}, 
		\frac{2\alpha - 3\alpha H + 2\gamma}{ \alpha  + 2H}, \tilde{C}_1 \right),
\]
%
\item[] $H = 1/2$, then
\[
		\sup_{s \in[0, T]} B_H(Y(s)) \in
		\mathcal{W}\left(\frac{2\alpha}{\alpha + 2H}, \tilde{\beta} , 
		\frac{2\gamma}{\alpha + 2H}, 2 \tilde{C}_2 \right),
\]
%
\item[] $H \in (1/2, 1]$, then
\[
		\sup_{s \in[0, T]} B_H(Y(s)) \in
    \mathcal{W}\left(\frac{2\alpha}{\alpha + 2H}, \tilde{\beta} , 
		\frac{2\gamma}{\alpha + 2 H}, \tilde{C}_2 \right),
\]
\end{itemize}
where
\begin{eqnarray*}
			\alpha &=& \min(\alpha_1, \alpha_2), \ \ \ \ \ \ 
      \tilde{\beta} = \beta^\frac{2H}{\alpha + 2H}
        \left(\frac{1}{2}\left(\frac{\alpha}{H}\right)^\frac{2H}{\alpha + 2H} +
        \left(\frac{H}{\alpha}\right)^\frac{\alpha}{\alpha + 2H} \right),		\\
      \tilde{C}_1 &=& \mathcal{H}_H\left(\frac{1}{2}\right)^\frac{1}{2H}
       \frac{C}{\sqrt{\alpha + 2H}}
       H^\frac{ \alpha + 6H + 2\gamma - 2}{2\alpha + 4H}
       (\alpha\beta)^\frac{1 - 2H - \gamma}{\alpha + 2H},\ \ \ \ \ \ 
       \tilde{C}_2 = \frac{C\sqrt{H}}{\sqrt{\alpha + 2H}}
       \left( \frac{H}{\alpha\beta}\right)^\frac{\gamma}{\alpha + 2H},
\end{eqnarray*}
\\
with $\beta, \gamma, C$ as in Corollary \ref{cor.exact}.
\end{Proposition}

The proof of Proposition \ref{th.fbm} is presented in Section \ref{th.fbm.proof}. \\

We now apply Proposition \ref{th.fbm} to calculate the exact asymptotics 
for the special case of iterated fractional Brownian motion process $\{B_{H_2}(B_{H_1}(t))\}$.

\begin{Proposition}	\label{prop.iter.fbm}
Let $\{B_{H_1}(t): t \in \R \}$ and $\{B_{H_2}(t): t \in [0,\infty)\}$ be independent
fractional Brownian motions with Hurst parameters $H_1, H_2 \in [0,1)$ respectively.
Then 
\[
		\sup_{s \in [0,T]}B_{H_2}(B_{H_1}(s)) \in \calW\left(\alpha, \beta, \gamma, C \right),
\]
where
\begin{eqnarray*}
	\alpha = \frac{2}{H_2 + 1}, \ \ \ \ 
	\beta = \left(\frac{1}{T}\right)^\frac{2H_1H_2}{1 + H_2}
					\left(\frac{1}{2}\right)^\frac{H_2}{1 + H_2}
					\left(\frac{1}{2} \left(\frac{2}{H_2}\right)^\frac{H_2}{1 + H_2} +
					\left(\frac{H_2}{2}\right)^\frac{1}{1 + H_2} \right),
\end{eqnarray*}
and 
\begin{eqnarray*}
		&&(\gamma, C)
	=
		\left\{ \begin{array}{lll}
			(\gamma_1, C_1) & for & H_1 \in (0, 1/2), H_2 \in (0, 1/2),  \vspace{0.5ex} \\
			(\gamma_2, 2C_2) & for & H_1 \in (0,1/2), H_2 = 1/2,	\vspace{0.5ex} \\ 
			(\gamma_2, C_2) & for & H_1 \in (0, 1/2), H_2 \in (1/2,1], \vspace{0.5ex} \\ 
			(\gamma_3, 2C_3) & for & H_1 = 1/2, H_2 \in (0, 1/2), \vspace{0.5ex} \\
			(\gamma_4, 4C_4) & for & H_1 = 1/2, H_2 = 1/2, \vspace{0.5ex} \\	
			(\gamma_4, 2C_4) & for & H_1 = 1/2, H_2 \in (1/2,1], \vspace{0.5ex} \\
			(\gamma_3, C_3) & for &	H_1 \in  (1/2,1], H_2 \in (0, 1/2), \vspace{0.5ex} \\
			(\gamma_4, 2C_4) & for & H_1 \in (1/2,1], H_2 = 1/2,	 \vspace{0.5ex} \\
			(\gamma_4, C_4) & for & H_1 \in  (1/2,1], H_2 \in (1/2,1],				
		\end{array} \right.
\end{eqnarray*}

%
%
%
%
%
%
%
%
%
with
\begin{eqnarray*}
	\gamma_1 &=& \frac{1 - H_1 -3H_1H_2}{H_1(1 + H_2)}, \ \ \ 
	\gamma_2 \ =\ \frac{1 - 3H_1}{H_1(1 + H_2)},		\	\ \ 
	\gamma_3 \ =\  \frac{1 - 3H_2}{1 + H_2},  \ \ \ 
	\gamma_4 \ =\  -\frac{1}{1 + H_2}, \\
		C_1 
	&=&  
		\left(\frac{1}{T}\right)^{H_2 - 3H_1H_2}
		\frac{\calH_{H_1}\calH_{H_2}}{H_1\sqrt{\pi(1 + H_2)}}
		\left(\frac{1}{2}\right)^{\frac{H_1 + H_2 + 2H_1H_2}{2H_1H_2}} 
		H_2^\frac{1 - 3H_1 + 3H_1H_2}{2H_1(1 + H_2)}, \\
		C_2 
	&=& 
		\left(\frac{1}{T}\right)^{H_2 - 3H_1H_2}
		\frac{\calH_{H_1}}{H_1\sqrt{\pi (1 + H_2)}}
		\left(\frac{1}{2}\right)^{\frac{1}{2H_1} + 1}
		H_2^\frac{1 - 2H_1 + H_1H_2}{2H_1(1 + H_2)},  \\
	C_3 &=&  
	T^{H_1H_2}\frac{\calH_{H_2}}{\sqrt{\pi (1 + H_2)}} 
	\left(\frac{1}{2}\right)^{\frac{1}{2H_2} - 1}
	H_2^\frac{3H_2 - 1}{2 + 2H_2},	\ \ \ \ \ \ 
		C_4 
	= 
		T^{H_1H_2}
		\frac{1}{2\sqrt{\pi(1 + H_2)}} H_2^\frac{H_2}{2(1 + H_2)}.
\end{eqnarray*}	
\end{Proposition}
\begin{proof}
Due to self-similarity of fBm
\[
		\pr\left(\sup_{s \in [0,T]}B_{H_2}(B_{H_1}(s)) > u \right)
	=
		\pr\left(\sup_{s \in [0,1]}B_{H_2}(B_{H_1}(s)) > \frac{u}{T^{H_1H_2}} \right).
\]
Moreover, due to Lemma 4.2 in \cite{ArD11} (see also \cite{Pit96}, Theorem D3) 
\[
	\sup_{s\in[0,1]}B_{H_1}(s) \in
    \mathcal{W}\left(2, \frac{1}{2}, \frac{1}{H_1} - 3, \frac{\mathcal{H}_{H_1}}{H_1\sqrt{\pi}} 2^{-\frac{H_1+1}{2H_1}}\right)
		\	{\rm for}\ H_1 \in (0, 1/2);
\]
\[
    \sup_{s\in[0,1]}B_{H_1}(s) \in
    \mathcal{W}\left(2, \frac{1}{2}, -1, \frac{2}{\sqrt{2\pi}} \right)
 \ {\rm for}\ H_1 = 1/2; 
\]
\[
    \sup_{s\in[0,1]}B_{H_1}(s) \in
    \mathcal{W}\left(2, \frac{1}{2}, -1, \frac{1}{\sqrt{2\pi}} \right)
\ {\rm for}\ H_1 \in (1/2,1].
\]
Additionally, by stationarity of the increments of fBm
\[
		-\inf_{s \in [0,1]} B_{H_1}(s) 
	\stackrel{d}{=}
		\sup_{s \in [0,1]} B_{H_1}(s).
\]
Now, in order to complete the proof it suffices to apply Proposition \ref{th.fbm}.
\end{proof}

\subsection{The stationary case} \label{sec.short.st}
In this section, we analyze the asymptotic behavior of \refs{main.short2} for the case of
$\{X(t): t \in [0, \infty)\}$ being a centered {\it stationary} Gaussian process with a.s. continuous sample paths
and covariance function $r(t) := \cov(X(s), X(s + t))$.
We impose the following assumptions on $r(t)$ (see, e.g., \cite{Pit96}): \\
\\
{\bf D1} \ \ $r(t) = 1 - C|t|^\alpha + o(|t|^\alpha)$ as $t \to 0$, with $\alpha \in (0,2]$ and $C>0$; \\
{\bf D2} \ \ $r(t) < 1$ for all $t > 0$.\\

In this case, we are able to give the exact form of the asymptotics for general class of 
stochastic processes $\{Y(t) : t \in [0, \infty)\}$ that are independent of $\{X(t)\}$, have a.s. continuous sample paths and finite 
average {\it span} over interval $[0,T]$. Therefore, we assume that \\ \\
{\bf S1}\ \ 
$
		\E\left[\sup_{s \in [0,T]} Y(s) - \inf_{s \in [0,T]} Y(s)\right] < \infty.
$	\\
\begin{Proposition}		\label{st.short}
Let $\{X(t) : t \in \R \}$ be a centered stationary Gaussian process with covariance function $r(t)$ that satisfies {\bf D1, D2} and
 $\{Y(t) : t \in [0, \infty)\}$ be an independent of $\{X(t)\}$ stochastic process with a.s. continuous sample paths that satisfies {\bf S1}.
Then
\[
		\pr\left(\sup_{s \in [0, T]} X(Y(s)) > u\right) 
	=	 
		\E (\calT) C^\frac{1}{\alpha} \mathcal{H}_\alpha u^\frac{2}{\alpha} \Psi(u) (1 + o(1)) 
\]
as $u \toi$, where 
$\calT = \sup_{s \in [0,T]} Y(s) - \inf_{s \in [0,T]} Y(s)$.
\end{Proposition}

\begin{proof}
Due to stationarity of the process $\{X(t)\}$, we have
\begin{eqnarray*}
		\pr\left(\sup_{t \in [0, T]} X(Y(t)) > u\right)
	&=&
		\pr\left(\sup_{t \in [\inf_{s \in [0, T]}Y(s),\ \sup_{s \in [0, T]}Y(s) ]} X(t) > u\right)	\\
	&=&
		\pr\left(\sup_{t \in [0, \calT]} X(t) > u\right).
\end{eqnarray*}
Now, in order to complete the proof it suffices to apply Theorem 3.1 in \cite{ArD12}.
\end{proof}

\begin{Rem}
Equivalently, Proposition \ref{st.short} states that
\begin{eqnarray}
		\pr\left(\sup_{t \in [0, T]} X(Y(t)) > u\right)
	\nonumber
	&=&
		\E(\calT)
		\, \pr\left(\sup_{t \in [0, 1]} X(t) > u\right) (1 + o(1))
\end{eqnarray}
as $u \toi$, where $\calT = \sup_{s \in [0,T]} Y(s) - \inf_{s \in [0,T]} Y(s)$.
\end{Rem}

\section{Long timescale case}		\label{sec.long.st}

In this section, we investigate
\begin{equation}	\label{main.long2}
		\lim_{u \toi} \pr\left(\sup_{s \in [0,h(u)]} X(Y(s)) >u \right)
\end{equation}
for a suitably chosen function $h(u)$. 

In order to formulate the results, it is convenient to introduce the notation
\[
		\sigma^{-1}(t) := \inf\{y \in [0,\infty) : \sigma(y) > t\}
\] 
for the generalized inverse of the function $\sigma(t)$.\\

We start with the observation that \refs{main.long2} can be straightforwardly obtained for any independent, self-similar processes
$\{X(t)\}$ and $\{Y(t)\}$ with a.s. continuous sample paths. 
\begin{Rem}
Let $\{X(t) : t \in \R\}$ and $\{Y(t): t \in [0, \infty)\}$ be independent, self-similar stochastic processes 
with a.s. continuous sample paths and self-similarity indexes $\lambda_X$ and $\lambda_Y$
respectively. Then, for
$h(u) = u^{1/\lambda_X\lambda_Y}(1 + o(1))$ as $u \toi$, we have
\[
		\lim_{u \toi}
		\pr\left(\sup_{t \in [0, h(u)]} X(Y(t)) > u\right) 
	= 
		\pr\left(\sup_{t \in [\inf_{s \in [0,1]} Y(s),\ \sup_{s \in [0,1]} Y(s)]}
		X(t) > 1\right).
\]
\end{Rem}

In the next theorem, we extend this observation to the case of $\{X(t) : t \in \R\}$ and $\{Y(t) : t \in [0,\infty)\}$ 
being two independent, centered Gaussian processes with stationary increments, a.s. continuous sample
paths, $X(0) = 0$ and $Y(0) = 0$ a.s., and variance functions $\sigma_X^2(t) := \var(X(t))$ and
$\sigma_Y^2(t) := \var(Y(t))$ respectively. 
We assume that variance functions of both processes satisfy
the following assumptions 
\\ \\
{\bf B1}\ \ $\sigma^2(t) \in C([0, \infty))$ is ultimately strictly increasing ;	\\
{\bf B2}\ \ $\sigma^2(t)$ is regularly varying at $\infty $ with parameter
$\alpha \in (0 , 2]$;\\
{\bf B3}\ \ $\sigma^2(t)$ is regularly varying at $0$ with parameter
$\beta \in (0 , 2]$. \\

In order to formulate the result, it is convenient to introduce the
notation
\[
		\mathcal{L}(\alpha_X, \alpha_Y) 
	=
		\pr\left(\sup_{t \in [\inf_{s \in [0,1]} 
		B_{\alpha_Y/2}(s),\ \sup_{s \in [0,1]} B_{\alpha_Y/2}(s)]}
		B_{\alpha_X/2}(t) > 1\right),
\]
where $\left\{B_{\alpha_X/2}(t)\right\}$, 
$\left\{B_{\alpha_Y/2}(t)\right\}$ are independent 
fractional Brownian motions with Hurst parameters $\frac{\alpha_X}{2}$ and 
$\frac{\alpha_Y}{2}$ respectively.

\begin{Theorem}		\label{long.st.incr}
Let $\{ X(t): t \in \R \}$ and
$\{Y(t): t \in [0, \infty)\}$ be independent, centered Gaussian processes with stationary increments that satisfy
{\bf B1 -- B3} with parameters $\alpha_X, \beta_X, \alpha_Y, \beta_Y$ respectively. Then, for
$h(u) = \sigma_Y^{-1}(\sigma_X^{-1}(u))(1 + o(1))$ as  $u \toi$, we have
\[
		\lim_{u \toi}
		\pr\left(\sup_{t \in [0, h(u)]} X(Y(t)) > u\right) 
	= 
		\mathcal{L}(\alpha_X, \alpha_Y).
\]

\end{Theorem}
The proof of Theorem \ref{long.st.incr} is presented in Section \ref{long.st.incr.proof}. \\

The second part of this section focuses on the analysis of limiting behavior of \refs{main.long2} in the case of $\{X(t) : t \in \R \}$ being
a centered stationary Gaussian process with a.s. continuous sample paths
and covariance function $r(t) := \cov(X(s), X(s + t))$ that satisfies
\\ \\
{\bf D1}\ \  $r(t) = 1 - C|t|^\alpha + o(|t|^\alpha)$ as $t \to 0$, with $\alpha \in (0,2]$ and $C>0$; \\
{\bf D2}\ \  $r(t) < 1$ for all $t > 0$;	\\
{\bf D3}\ \  $r(t)\log(t) \to r$ as $t \to \infty$, with $r \in [0, \infty)$.
\\ 

We study \refs{main.long2} for both weakly and strongly dependent stationary Gaussian processes, i.e.,
for $r = 0$ and $r > 0$ respectively. We refer to \cite{Tan12, Tan13} for recent results on
asymptotic behavior of supremas of strongly dependent Gaussian processes. 

In this settings, in Theorem \ref{long.stat.main}, we provide \refs{main.long2} in the case of $\{Y(t): t \in [0, \infty)\}$ 
being a centered Gaussian process with stationary increments and 
variance function $\sigma_Y^2(t)$ that satisfies conditions {\bf B1 -- B3}.
Moreover, in Proposition \ref{long.st.ss}, we analyze \refs{main.long2} for self-similar process $\{Y(t)\}$ that is not necessarily Gaussian.

\begin{Theorem}		\label{long.stat.main}
Let $\{X(t): t \in \R\}$ be a centered stationary Gaussian process with covariance function 
that satisfies {\bf D1 -- D3} and $\{Y(t): t \in [0,\infty) \}$ be an independent of $\{X(t)\}$ Gaussian process  with 
a.s. continuous sample paths, stationary increments, and variance function $\sigma^2_Y(t)$ that satisfies {\bf B1 -- B3}
with parameters $\alpha_Y, \beta_Y$. Then, for
$
	h(u) 
		= 
	\sigma_Y^{-1} \left(\left(C^\frac{1}{\alpha} \mathcal{H}_\alpha 
					u^\frac{2}{\alpha} \Psi(u)\right)^{-1}\right) (1 + o(1))$ as $u \toi$,

		we have
\[
	\lim_{u \toi}
	\pr\left(\sup_{s \in [0, h(u)]} X(Y(s)) > u\right) 
=	
	1 - \E\exp\left(- \calT \exp(-r + \sqrt{2r} \calN)\right),
\]
	where 
$\calT = \sup_{s \in [0, 1]} B_{\alpha_Y/2}(s) - \inf_{s \in [0, 1]} B_{\alpha_Y/2}(s)$
and $\calN$ is a normal random variable independent of $\calT$.
\end{Theorem}
The proof of Theorem \ref{long.stat.main} is given in Section \ref{long.stat.main.proof}.

\begin{Proposition}		\label{long.st.ss}
Let $\{X(t) : t \in \R \}$ be a stationary Gaussian process with covariance function that satisfies {\bf D1 -- D3} and
 $\{Y(t) : t \in [0,\infty)\}$ be a self-similar stochastic process with parameter $\lambda_Y$,  independent of the process 
$\{X(t)\}$. Then, for 
$
	h(u) 
		= 
	\left[C^\frac{1}{\alpha} \mathcal{H}_\alpha u^\frac{2}{\alpha} \Psi(u) \right]^{-1/\lambda_Y} (1 + o(1))$ as $u \toi$, 
			we have
\[
		\lim_{u \toi}
		\pr\left(\sup_{s \in [0, h(u)]} X(Y(s)) > u\right) 
	=	 
		1 - \E\exp\left(- \calT \exp\left(-r + \sqrt{2r} \calN\right)\right),
\]
where 
$\calT = \sup_{s \in [0,1]} Y(s) - \inf_{s \in [0,1]} Y(s)$ and $\calN$ is a normal random variable 
independent of $\calT$.
\end{Proposition}
\begin{proof}
Due to stationarity of the process $\{X(t)\}$ and self-similarity of the process $\{Y(t)\}$, we have
\begin{eqnarray}
		\pr\left(\sup_{t \in [0, h(u)]} X(Y(t)) > u\right)
	\nonumber
	&=&
		\pr\left(\sup_{t \in [0, \sup_{s \in [0, h(u)]}Y(s) - \inf_{s \in [0, h(u)]}Y(s)]} X(t) > u\right)\\
	\nonumber
	&=&	
		\pr\left(\sup_{t \in [0, (h(u))^{\lambda_Y} \calT]} X(t) > u\right) \\
	&=&		\label{eq.long.random}
		1 - \E\exp\left(-\calT \exp\left(-r + \sqrt{2r}\calN\right)\right),
\end{eqnarray}
where \refs{eq.long.random} follows by the reasoning as in the proof of Theorem \ref{long.stat.main}. 

\end{proof}
\begin{Rem}
Note, that in the case of weakly dependent stationary Gaussian process $\{X(t)\}$, that is, if $r = 0$ in {\bf D3}, we obtain
the following result 
\[
		\lim_{u \toi}
		\pr\left(\sup_{s \in [0, h(u)]} X(Y(s)) > u\right) 
	=
		1 - \E e^{- \calT},
\]
where in the setting of Theorem \ref{long.stat.main},
$h(u) = \sigma_Y^{-1} \left(\left(C^\frac{1}{\alpha} \mathcal{H}_\alpha 
					u^\frac{2}{\alpha} \Psi(u)\right)^{-1}\right) (1 + o(1))
$
as $u \toi$, and $\calT = \sup_{s \in [0, 1]} B_{\alpha_Y/2}(s) - \inf_{s \in [0, 1]} B_{\alpha_Y/2}(s)$;\\
and in the setting of Proposition \ref{long.st.ss},
$h(u) = \left[C^\frac{1}{\alpha} \mathcal{H}_\alpha u^\frac{2}{\alpha} \Psi(u) \right]^{-1/\lambda_Y}(1 + o(1))$
as $u \toi$, and $\calT = \sup_{s \in [0,1]} Y(s) - \inf_{s \in [0,1]} Y(s)$.
\end{Rem}
\section{Proofs}
In this section, we present detailed proofs of Theorem \ref{th.short}, Proposition \ref{th.fbm}, 
Theorem \ref{long.st.incr} and Theorem \ref{long.stat.main}.
\subsection{Proof of Theorem \ref{th.short}}		\label{th.short.proof}
In view of inclusion -- exclusion principle 
\begin{equation}	\label{in.ex}
		\pr\left(\sup_{s \in [0,T]} X(Y(s)) > u\right)
	=
		P_1(u) + P_2(u) - P_3(u),	
\end{equation}
where
\[
	P_1(u) = \pr\left(\sup_{s \in [-\calK, 0]} X(s) > u\right), \ \ \ \ \ \ 
	P_2(u) = \pr\left(\sup_{s \in [0,\calM]} X(s) > u\right),
\]
\[
	P_3(u) = \pr\left(\sup_{s \in [-\calK,0]} X(s) > u, \sup_{s \in [0,\calM]} X(s) > u\right).
\]
Observe that by definition of the process $\{X(t)\}$,
\begin{equation}		\label{eq.symetry}	
P_1(u)
	=
	 \pr\left(\sup_{s \in [0, \calK]} X(s) > u\right).
\end{equation}

The case (i) is a consequence of the fact that, by \refs{eq.symetry} and Theorem 3.1 in \cite{ArD11}, 
$\pr(\calK > u ) =  o(\pr(\calM > u))$ implies
$P_1(u) = o(P_2(u))$. Thus, 
\[	
	P_2(u) \le	\pr\left(\sup_{s \in [0,T]} X(Y(s)) > u\right) \le P_1(u) + P_2(u) = P_2(u) (1 + o(1))
\] 
as $u \toi$, which in view of Theorem 3.1 in \cite{ArD11}, completes the proof for the case (i).
A similar reasoning implies that for the case (ii), we have
\begin{eqnarray*}	
		P_1(u) \le	\pr\left(\sup_{s \in [0,T]} X(Y(s)) > u\right) \le P_1(u) + P_2(u) = P_1(u) (1 + o(1))
\end{eqnarray*}
as $u \toi$, which in view of Theorem 3.1 in \cite{ArD11}, completes the proof for the case (ii).\\
In order to prove (iii), without loss of generality, we assume that
\begin{equation}	\label{as.sup.Y}
		\pr(\calM > u) \ge \pr(\calK > u) (1 + o(1))
\end{equation}
as $u \toi$. Due to \refs{in.ex} combined with \refs{eq.symetry} and Theorem 3.1 in \cite{ArD11}, it suffices to show that $P_3(u)$ is negligible. We distinguish the case $\calK \le \calM$ and the case $\calK > \calM$ and obtain
\begin{eqnarray}
		P_3(u) 
	\nonumber
	&=& 
		\pr\left(\sup_{s \in [-\calK, 0]} X(s) > u, \sup_{s \in [0, \calM]} X(s) > u, \calK \le \calM\right) \\
	\nonumber
	&+&
		\pr\left(\sup_{s \in [-\calK, 0]} X(s) > u, \sup_{s \in [0, \calM]} X(s) > u, \calK > \calM\right) \\
	\nonumber
	&\le&	
		\pr\left(\sup_{s \in [-\calM, 0]} X(s) > u, \sup_{s \in [0, \calM]} X(s) > u\right) \\
	\nonumber
	&+&
		\pr\left(\sup_{s \in [-\calK, 0]} X(s) > u, \sup_{s \in [0, \calK]} X(s) > u\right) \\
	&\le&	\label{double.equal}
		2\pr\left(\sup_{s \in [-\calM, 0]} X(s) > u, \sup_{s \in [0, \calM]} X(s) > u\right) (1 + o(1))
\end{eqnarray}
as $u \toi$, where \refs{double.equal} is due to the assumption \refs{as.sup.Y}.

To find an upper bound of \refs{double.equal} it is convenient to make the following decomposition
\begin{eqnarray*}
\lefteqn{
		\pr\left(\sup_{s \in [-\calM, 0]} X(s) > u, \sup_{s \in [0, \calM]} X(s) > u\right) }\\
	&=&
		\left(\int_0^{a(u)} + \int_{a(u)}^{A(u)} + \int_{A(u)}^\infty \right) 
		\pr\left(\sup_{s \in [-w, 0]} X(s) > u, \sup_{s \in [0, w]} X(s) > u\right)  dF_\calM(w)\\
	&=&
		I_1 + I_2 + I_3,	
\end{eqnarray*}		
		
where
\begin{equation}	\label{au_Au}
		a(u) = u^\frac{2}{\alpha_\infty + 2 \alpha_1},
		A(u) = u^\frac{4}{2\alpha_\infty + \alpha_1}.
\end{equation}
Let $\varepsilon > 0$. We analyze each of the integrals $I_1, I_2, I_3$ separately.
\\
{\it Integral $I_1$:}
\begin{eqnarray}
		I_1
	\nonumber
	&\le&
		\int_0^{a(u)} \pr\left(\sup_{s \in [0, w]} X(s) > u\right) dF_\calM(w) \\
	\nonumber
	&\le&
		\pr\left(\sup_{s \in [0, a(u)]} X(s) > u\right)	\\
	&\le&		\label{ineq.sup.long}
		{\rm Const}\, a(u) \left(\frac{u}{\sigma_X(a(u))}\right)^\frac{2}{\alpha_\infty} \Psi\left(\frac{u}{\sigma_X(a(u))}\right)	\\
	\nonumber
	&\le&
		\exp\left(-u^{\frac{2\alpha_1}{\alpha_1 + \alpha_\infty} + \varepsilon}\right) (1 + o(1))
\end{eqnarray}
as $u \toi$, where \refs{ineq.sup.long} is due to (16) in \cite{ArD11} 
(see also the proof of Lemma 6.3 in \cite{ArD11}).	\\
{\it Integral $I_3$:}
\begin{eqnarray*}
		I_3
	&\le&
		\int_{A(u)}^\infty \pr\left(\sup_{s \in [0, w]} X(s) > u\right) dF_\calM(w)	\\
	&\le&
		\pr(\calM > A(u))	\\
	&=&
		C_1 (A(u))^{\gamma_1} \exp\left(-\beta_1 (A(u))^{\alpha_1} \right)(1 + o(1))	\\
	&\le&
		\exp\left(-u^{\frac{2\alpha_1}{\alpha_1 + \alpha_\infty} + \varepsilon}\right) (1 + o(1))
\end{eqnarray*}
as $u \toi$.		\\
The above, combined with the observation that for each $\eta > 0$ and sufficiently large $u$,
\begin{eqnarray*}
		\pr(X(\calM) > u)
	&=&
		\pr\left(\sigma_X(\calM) \calN > u\right)
	\ge
		\pr\left(\sigma_X(\calM) > u^\frac{\alpha_\infty}{\alpha_1 + \alpha_\infty}\right)
		\pr\left(\calN > u^\frac{\alpha_1}{\alpha_1 + \alpha_\infty}\right)	\\
	&\ge&
		\exp\left(- u^{\frac{2\alpha_1}{\alpha_1 + \alpha_\infty} + \eta}\right),
\end{eqnarray*}
leads to the conclusion that $I_1$ and $I_3$ are negligible. \\
{\it Integral $I_2$}: 
Observe that, due to {\bf A1}, $\sigma_X^2(|t|) \le \sigma_X^2(|t - s|)$, for each $(s,t) \in [-w, 0]\times[0,w]$. Hence
\begin{equation}		\label{var.bound}
		\var(X(s) + X(t))
	=
		2\sigma_X^2(|s|) + 2\sigma_X^2(|t|) - \sigma_X^2(|t - s|)
	\le	
		 3\sigma_X^2(w),
\end{equation}
for $(s, t) \in [-w, 0] \times [0,w]$.
Thus, according to the Borell inequality (see, e.g., \cite{Adl90}, Theorem 2.1), 
combined with \refs{var.bound}, $I_2$ is bounded by
\begin{eqnarray}
	\nonumber
	\lefteqn{
		\int_{a(u)}^{A(u)} \pr\left(\sup_{(s, t) \in [-w, 0] \times [0, w]} \left[ X(s) + X(t) \right] > 2u\right)  dF_\calM(w)}	\\
	&\le&		\label{bound.Borell}
		2 \int_{a(u)}^{A(u)} 
		\exp\left(- \frac{2u^2}{3\sigma_X^2(w)} 
					\left(1 - \frac{1}{2u} \E \left(\sup_{(s, t) \in [-w, 0] \times [0, w]} \left[X(s) + X(t)\right]\right)\right)^2 \right)		
					dF_\calM(w).
\end{eqnarray}
Moreover,
\begin{equation}	\label{mean.ineq}
		0 
	\le 
		\E \left(\sup_{(s, t) \in [-w, 0] \times [0, w]} \left[X(s) + X(t)\right]\right)
	\le
		2 \E \left(\sup_{s \in [0,w]} X(t)\right).
\end{equation}
To find the upper bound of $\E \sup_{t \in [0,w]} X(t)$, we use metric entropy method (see, e.g., \cite{Lif12},
Chapter 10).
At the beginning, for any ${\bb T} \subseteq \R$ define the semimetric	\\
\[
		d(s, t) := \sqrt{\E |X(t) - X(s)|^2} = \sigma_X(|t - s|).
\]
The metric entropy ${\bb H}_d(\bb T, \epsilon)$ is defined as $\log N_d(\bb T, \epsilon)$, 
where $N_d(\bb T, \epsilon)$
denotes the minimal number of points in an $\epsilon$-net in $\bb T$ with respect to the semimetric $d$.

Observe that for ${\bb T} = [0, w]$,
\[
		N_d(\bb T, \epsilon) 
	\le
		\frac{2w}{\sigma_X^{-1}(\epsilon)},
\]
which, in view of Theorem 10.1 in \cite{Lif12}, implies that
\begin{eqnarray}
		\E \sup_{t \in [0, w]} X(t)
	\nonumber
	&\le&
		4 \sqrt{2} \int_0^{\sigma_X(w)} \sqrt{\log \frac{2w}{\sigma_X^{-1}(\epsilon)}} \, d\epsilon	\\
	&\le&		\label{entr.bound1}
		4 \sqrt{2} \int_0^{\sqrt{D}w^\frac{\alpha_\infty}{2}} 
			\sqrt{\log \frac{2D^\frac{1}{\alpha_\infty} w}{\epsilon^\frac{2}{\alpha_\infty}}} \, d\epsilon	\\
	&=&			\label{entr.bound2}
		4 \sqrt{2} \int_{1/w}^\infty \frac{\sqrt{D} \alpha_\infty}{2} x^{-\frac{\alpha_\infty}{2} - 1}
		\sqrt{\log 2wx} \, dx	\\
	\nonumber
	&\le&		
		4 \alpha_\infty \sqrt{Dw} \int_{1/w}^\infty x^{-\frac{\alpha_\infty + 1}{2}}
		\, dx	\\
	&\le&			\label{entropy.bound}
		Bw^\frac{\alpha_\infty}{2},	
\end{eqnarray}
where $B = \frac{8\sqrt{D} \alpha_\infty}{\alpha_\infty -1}$, \refs{entr.bound1} is due to {\bf A3},
and \refs{entr.bound2} is by substitution $x := D^{1/\alpha_\infty} \epsilon^{-2/\alpha_\infty}$.\\
Finally, due to \refs{mean.ineq} combined with \refs{entropy.bound} and \refs{au_Au}
\[
		0 
	\le 
		\frac{1}{2u} \E \left(\sup_{(s, t) \in [-w, 0] \times [0, w]} \left[X(s) + X(t)\right] \right)
	\le 
		 B\frac{w^\frac{\alpha_\infty}{2}}{u}
	\le 
		B u^{-\frac{\alpha_1}{2\alpha_\infty + \alpha_1}},
\]
for each $w \in [a(u), A(u)]$, which implies that
\begin{eqnarray*}		
		\left(1 - \frac{1}{2u} \E \left(\sup_{(s, t) \in [-w, 0] \times [0, w]} \left[X(s) + X(t)\right]\right) \right)^2
	\to 
		1
\end{eqnarray*}
as $u \toi$, uniformly for $w \in [a(u), A(u)]$, and hence
\begin{equation}		\label{o_small}
		\exp\left(- \frac{2u^2}{3\sigma_X^2(w)} 
				\left(1 - \frac{1}{2u} \E \left(\sup_{(s, t) \in [-w, 0] \times [0, w]} \left[X(s) + X(t)\right] \right)\right)^2 \right)
	=
		o\left(\Psi\left(\frac{u}{\sigma_X(w)}\right)\right)
\end{equation}
as $u \toi$, uniformly for $w \in [a(u), A(u)]$. Thus, combining \refs{bound.Borell} with \refs{o_small}, we obtain, 
for sufficiently large $u$, the following
upper bound, 
\begin{eqnarray}
		I_2
	\nonumber
	\le
		2 \varepsilon \int_{a(u)}^{A(u)} \Psi\left(\frac{u}{\sigma_X(w)}\right)	dF_\calM(w)
	\nonumber
	\le		
		2\varepsilon \pr\left(\sup_{s \in [0,\calM]} X(s) > u\right),
\end{eqnarray}
which in view of Theorem 3.1 in \cite{ArD11}, implies that 
\[
		\limsup_{u \toi} \frac{I_2}{\pr\left(X(\calM) > u\right)}
	\le 
		2\varepsilon.
\]
In order to complete the proof it suffices to pass with $\varepsilon \to 0$.
\Halmos

\subsection{Proof of Proposition \ref{th.fbm}}		\label{th.fbm.proof}
The idea of the proof is analogous to the proof of Theorem \ref{th.short}, thus we present only
main steps of the argumentation. 
In view of inclusion -- exclusion principle 
\begin{equation}	\label{in.ex.BH}
		\pr\left(\sup_{s \in [0,T]} B_H(Y(s)) > u\right)
	=
		P_1(u) + P_2(u) - P_3(u),	
\end{equation}
where
$P_1(u) = \pr\left(\sup_{s \in [-\calK, 0]} B_H(s) > u\right)$,
$P_2(u) = \pr\left(\sup_{s \in [0,\calM]} B_H(s) > u\right)$,\\
$P_3(u) = \pr\left(\sup_{s \in [-\calK,0]} B_H(s) > u, \sup_{s \in [0,\calM]} B_H(s) > u\right)$.\\
Moreover observe that 
\begin{equation} \label{eq.symetry.fbm}
		P_1(u)
	=
		\pr\left(\sup_{s \in [0,\calK]} B_H(s) > u\right).
\end{equation}
Since the arguments for the cases $\pr(\calK > u) = o(\pr(\calM > u))$ as $u \toi$, and 
$\pr(\calM > u) = o(\pr(\calK > u))$ as $u \toi$
are similar to those in the proof of Theorem \ref{th.short}, then we focus on the case 
$\pr(\calK > u) = \frac{C_2}{C_1}\pr(\calM > u)(1 + o(1))$ as $u \toi$.\\
Without loss of generality, we assume that
\begin{eqnarray*}	
		\pr(\calM > u) \ge \pr(\calK > u) (1 + o(1))
\end{eqnarray*}
as $u \toi$. Due to \refs{in.ex.BH} combined with \refs{eq.symetry.fbm} and Theorem 4.1 in \cite{ArD11},
it suffices to show that $P_3(u)$ is negligible. In an analogous way to \refs{double.equal}, we obtain 
the following upper bound 
\begin{eqnarray*}
		P_3(u) 
	&\le&	\label{double.equal.BH}
		2\pr\left(\sup_{s \in [-\calM, 0]} B_H(s) > u, \sup_{s \in [0, \calM]} B_H(s) > u\right) (1 + o(1))
\end{eqnarray*}
as $u \toi$.
Then, we consider decomposition
\begin{eqnarray*}
\lefteqn{
		\pr\left(\sup_{s \in [-\calM, 0]} B_H(s) > u, \sup_{s \in [0, \calM]} B_H(s) > u\right) }\\
	&=&
		\left(\int_0^{a(u)} + \int_{a(u)}^{A(u)} + \int_{A(u)}^\infty \right) 
		\pr\left(\sup_{s \in [-w, 0]} B_H(s) > u, \sup_{s \in [0, w]} B_H(s) > u\right)  dF_\calM(w)\\
	&=&
		I_1 + I_2 + I_3,	
\end{eqnarray*}		
		
where
\begin{equation} \label{aa.fbm}
		a(u) = u^\frac{1}{H + \alpha_1},	\	\	\	\	
		A(u) = u^\frac{4}{4H + \alpha_1}.
\end{equation}		
Let $\varepsilon > 0$. We investigate the asymptotic behavior of each of the integrals.
\\
{\it Integral $I_1$:}
Due to self-similarity of $\{B_H(t)\}$ combined with Lemma 4.2 in \cite{ArD11}, we have, as $u \toi$,
\begin{eqnarray}
		I_1
	\nonumber
	\le
		\pr\left(\sup_{s \in [0, a(u)]} B_H(s) > u\right)
	=	
		\pr\left(\sup_{s \in [0,1]} B_H(s) > \frac{u}{(a(u))^H}\right)
	\le
		\exp\left(-u^{\frac{2\alpha_1}{\alpha_1 + H} + \varepsilon}\right) (1 + o(1)).
\end{eqnarray}
{\it Integral $I_3$:} We have, as $u \toi$,
\begin{eqnarray*}
		I_3
	\le
		\pr(\calM > A(u))
	\le
		\exp\left(-u^{\frac{2\alpha_1}{\alpha_1 + H} + \varepsilon}\right) (1 + o(1)).
\end{eqnarray*}
Observe that, due to Theorem 4.1 in \cite{ArD11}, for each $\eta > 0$ and sufficiently large $u$,
\begin{eqnarray*}
		\pr\left(\sup_{s \in [0,\calM]} B_H(s) > u\right)
	&\ge&
		\exp\left(- u^{\frac{2\alpha_1}{\alpha_1 + H} + \eta}\right) (1 + o(1))
\end{eqnarray*}
as $u \toi$. Thus, we conclude that $I_1$ and $I_3$ are negligible. \\
{\it Integral $I_2$}:
Observe that $|t|^{2H} \le |t - s|^{2H}$, for each 
$(s, t) \in [-w, 0] \times [0,w]$. Hence
\begin{equation}		\label{var.bound.BH}
		\var(B_H(s) + B_H(t))
	=
		2|s|^{2H} + 2|t|^{2H} - |t-s|^{2H}
	\le
		 3w^{2H}.
\end{equation}
for $(s, t) \in [-w, 0] \times [0,w]$. Thus, according to the Borell inequality (see, e.g., \cite{Adl90}, Theorem 2.1), 
combined with \refs{var.bound.BH}, $I_2$ is bounded by
\begin{eqnarray}
	\nonumber
	\lefteqn{
		\int_{a(u)}^{A(u)} \pr\left(\sup_{(s, t) \in [-w, 0] \times [0, w]} \left[ B_H(s) + B_H(t) \right] > 2u\right)  dF_\calM(w)}	\\
	&\le&		\label{bound.Borell.BH}
		2 \int_{a(u)}^{A(u)} 
		\exp\left(- \frac{2u^2}{3w^{2H}} 
					\left(1 - \frac{1}{2u} \E \left(\sup_{(s, t) \in [-w, 0] \times [0, w]} \left[B_H(s) + B_H(t)\right]\right)\right)^2 \right)		
					dF_\calM(w).
\end{eqnarray}
Moreover, due to self-similarity of $\{B_H(t)\}$  
\[
		0 
	\le 
		 \E \left(\sup_{(s, t) \in [-w, 0] \times [0, w]} \left[B_H(s) + B_H(t)\right]\right)	
	\le
		2 \E \left(\sup_{s \in [0,w]} B_H(s)\right)
	=
		Bw^H,
\]
where
$ B = 2\E \sup_{s \in [0,1]} B_H(s) $,
which due to \refs{aa.fbm}, implies that
\[
		\left(1 - \frac{1}{2u} \E \left(\sup_{(s, t) \in [-w, 0] \times [0, w]} \left[B_H(s) + B_H(t)\right]\right) \right)^2
	\to 
		1
\]
as $u \toi$, uniformly for $w \in [a(u), A(u)]$, and hence
\begin{equation}	\label{o_small.BH}
		\exp\left(- \frac{2u^2}{3w^{2H}} 
				\left(1 - \frac{1}{2u} \E \left(\sup_{(s, t) \in [-w, 0] \times [0, w]} \left[B_H(s) + B_H(t)\right] \right)\right)^2 \right)
	=
		o\left(\Psi\left(\frac{u}{w^H}\right)\right)
\end{equation}
as $u \toi$, uniformly for $w \in [a(u), A(u)]$. Thus, combining \refs{bound.Borell.BH} with \refs{o_small.BH},
we obtain, for sufficiently large $u$, the following
upper bound, 
\begin{eqnarray*}
		I_2
	\le
		2 \varepsilon \int_{a(u)}^{A(u)} \Psi\left(\frac{u}{\sigma_X(w)}\right)	dF_\calM(w)
	\le
		2\varepsilon \pr\left(\sup_{s \in [0,\calM]} X(s) > u\right).
\end{eqnarray*}
which in view of Theorem 4.1 in \cite{ArD11}, implies that 
\[
		\limsup_{u \toi} \frac{I_2}{\pr\left(X(\calM) > u\right)}
	\le 
		2\varepsilon.
\]
In order to complete the proof it suffices to pass with $\varepsilon \to 0$.
\Halmos
\subsection{Proof of Theorem \ref{long.st.incr}} \label{long.st.incr.proof}
In further analysis we use the following notation 
\[
		X_{\sigma_Y(h(u))}(s) := \frac{X(\sigma_Y(h(u))s)}{\sigma_X(\sigma_Y(h(u)))} \ \ \ \ {\rm and} \ \ \ \ 
		Y_{h(u)}(s) := \frac{Y(h(u)s)}{\sigma_Y(h(u))}.
\]
Moreover, we denote
\begin{eqnarray*}
		\calV_u := \inf_{s \in [0,1]} Y_{h(u)}(s), \ \ \ \ 
		\calW_u := \sup_{s \in [0,1]} Y_{h(u)}(s),
\end{eqnarray*}
\begin{eqnarray*}		
		\calV := \inf_{s \in [0,1]} B_{\alpha_Y/2}(s) \ \ \ \ 
		\calW := \sup_{s \in [0,1]} B_{\alpha_Y/2}(s).
\end{eqnarray*}
Let $\varepsilon > 0$ and $0 < A_\infty < \infty$. We start with the observation that $\lim_{u \toi} h(u) = \infty$, 
which also implies that $\lim_{u \toi} \sigma_Y(h(u)) = \infty$. Hence, 
due to Lemma 5.2 in \cite{BDZ04}
\begin{equation}	\label{weak.double}
		\left(\calV_u, \calW_u\right) \Rightarrow (\calV, \calW)\ \ \ {\rm as}\ u \toi
\end{equation}		
and
\begin{equation}		\label{weak.sup.fbm}
		\sup_{s \in [v, w]} X_{\sigma_Y(h(u))}(s) \Rightarrow \sup_{s \in [v, w]} B_{\alpha_X/2}(s)\ \ \ {\rm as}\ u \toi,
\end{equation}
uniformly for $(v,w) \in [-A_\infty, 0]\times[0, A_\infty]$, where $\Rightarrow$ denotes convergence in distribution.\\
By continuity of the sample paths of the processes $\{X(t)\}$ and $\{Y(t)\}$,
\begin{eqnarray}
\nonumber
\lefteqn{	
	\pr\left(\sup_{t \in [0, h(u)]} X(Y(t)) > u\right)}\\ 
	\nonumber
	&=&
	\pr\left(\sup_{t \in [\inf_{s \in [0, h(u)]}Y(s), 
	\sup_{s \in [0,h(u)]} Y (s)]} X(t) > u\right)  \\	
	\nonumber
	&=&
	\pr\left(\sup_{t \in [\calV_u, \calW_u]} X(\sigma_Y(h(u)) t) > u\right)\\
	&=&			\label{comb1}
  \pr\left(\sup_{t \in [\calV_u, \calW_u]} 
	X_{\sigma_Y(h(u))}(t) > \frac{u}{\sigma_X(\sigma_Y(h(u)))}\right).
\end{eqnarray}
To find an upper bound of \refs{comb1} we consider the following decomposition
\begin{eqnarray}
	\nonumber
	\lefteqn{
		\pr\left(\sup_{t \in [\calV_u, \calW_u]} 
		X_{\sigma_Y(h(u))}(t) > \frac{u}{\sigma_X(\sigma_Y(h(u))}\right) }\\
	\nonumber
	&\le&
		\left(\int_{-\infty}^{-A_\infty} \int_0^\infty +
		\int_{-A_\infty}^0 \int_0^{A_\infty} +
		\int_{-\infty}^0 \int_{A_\infty}^\infty\right) 
		\pr\left(\sup_{t \in [v, w]}  X_{\sigma_Y(h(u))}(t) > \frac{u}{\sigma_X(\sigma_Y(h(u)))}\right) d_{(\calV_u, \calW_u)}(v,w)	\\
	\nonumber
	&=&		
		I_1 + I_2 + I_3.
\end{eqnarray}
We analyze each of the integrals $I_1$, $I_2$, $I_3$ separately.\\
{\it Integral $I_1$}: Due to \refs{weak.double}, for sufficiently large $u$,
\begin{eqnarray*}
		I_1
	\le
		\pr\left(\calV_u \le -A_{\infty}\right)
	\le		
		(1 + \varepsilon) \pr\left(\calV \le -A_\infty\right).
\end{eqnarray*}
{\it Integral $I_3$}: Due to \refs{weak.double}, for sufficiently large $u$,
\begin{eqnarray*}
		I_3
	\le
		\pr\left(\calW_u > A_{\infty}\right)
	\le 
		(1 + \varepsilon) \pr\left(\calW > A_\infty\right).
\end{eqnarray*}
{\it Integral $I_2$}: For $u$ sufficiently large,
\begin{eqnarray}
		I_2
	\nonumber	
	&=&		
		\int_{-A_\infty}^0 \int_0^{A_\infty}
		\pr\left(\sup_{t \in [v, w]}  X_{\sigma_Y(h(u))}(t) > \frac{u}{\sigma_X(\sigma_Y(h(u)))}\right) d_{(\calV_u, \calW_u)}(v,w)	\\
	&\le&		\label{I1.weak.sup.fbm}
		(1 + \varepsilon) 
		\int_{-A_\infty}^0 \int_0^{A_\infty}
		\pr\left(\sup_{t \in [v, w]}  B_{\alpha_X/2}(t) > 1\right) d_{(\calV_u, \calW_u)}(v,w)	\\
	&\le&		\label{I1.weak.double}
		(1 + \varepsilon)^2 
		\int_{-A_\infty}^0 \int_0^{A_\infty}
		\pr\left(\sup_{t \in [v, w]}  B_{\alpha_X/2}(t) > 1\right) d_{(\calV, \calW)}(v,w)\\
	\nonumber
	&\le&		
		(1 + \varepsilon)^2 
		\pr\left(\sup_{t \in [\calV, \calW]}  B_{\alpha_X/2}(t) > 1 \right),
\end{eqnarray}

where \refs{I1.weak.sup.fbm} is due to \refs{weak.sup.fbm} and the fact that 
$\lim_{u \toi} \frac{u}{\sigma_X(\sigma_Y(h(u)))} = 1$, and \refs{I1.weak.double} is due to \refs{weak.double}, 
and the observation that 
$ \pr\left(\sup_{t \in [v, w]}  B_{\alpha_X/2}(t) > 1\right) $
is bounded and continuous function with respect to $(v,w)$. 
Thus, for each $\varepsilon > 0$, $A_\infty > A_0 > 0$,
\begin{eqnarray*}
		\limsup_{u \toi} \pr\left(\sup_{t \in [0, h(u)]} X(Y(t)) > u\right)
	&\le&
		(1 + \varepsilon)^2 \pr\left(\sup_{t \in [\calV, \calW]}B_{\alpha_X/2}(t) > 1\right)	\\
		&+& (1 + \varepsilon)\pr(\calV \le -A_\infty)
		+ (1 + \varepsilon)\pr(\calW > A_\infty).
\end{eqnarray*}
Analogously,
\[
		\liminf_{u \toi} \pr\left(\sup_{t \in [0, h(u)]} X(Y(t)) > u\right)
	\ge
		(1 - \varepsilon)^2 \pr\left(\sup_{t \in [\calV, \calW]}B_{\alpha_X/2}(t) > 1\right).
\]
In order to complete the proof it suffices to pass with $A_0 \to 0, A_\infty \toi$, and $\varepsilon \to 0$. 
\Halmos
\subsection{Proof of Theorem \ref{long.stat.main}}	\label{long.stat.main.proof}
In further analysis we use the following notation
\[
		\calT_{u} := \sup_{s \in [0,1]} Y_{h(u)}(s) 
		- \inf_{s \in [0,1]} Y_{h(u)}(s), \ \ \ \ \ 
{\rm where}	\	\	\	\	\	
		Y_{h(u)}(s) := \frac{Y(h(u)s)}{\sigma_{Y}(h(u))}.	
\]
Let $\varepsilon > 0$ and $0 < A_0 < A_\infty < \infty$. Note that due to Lemma 5.2 in \cite{BDZ04}
\begin{equation}	\label{weak.conv}
		\calT_{u} \Rightarrow \calT \ \ {\rm as} \ u \toi,
\end{equation}
where $\Rightarrow$ denotes convergence in distribution.	\\
It is convenient to consider the following decomposition
\begin{eqnarray*}	
		\pr\left(\sup_{s \in [0, h(u)]} X(Y(s)) > u\right)
	&=&
		\pr\left(\sup_{s \in \left[0,\calT_u\sigma_Y(h(u))\right]} X(s) > u\right)	\\
	&=& 
		\left(\int_0^{A_0} + \int_{A_0}^{A_\infty} + \int_{A_\infty}^\infty\right)	
		\pr\left(\sup_{s \in [0,t\sigma_Y(h(u))]} X(s) > u\right) dF_{\calT_u}(t)	\\
	&=&
		I_1 + I_2 + I_3.
\end{eqnarray*}
We analyze each of the integrals $I_1, I_2, I_3$ separately.\\
{\it Integral $I_1$}:
Due to Lemma 3.3 in \cite{Tan13}, for sufficiently large $u$,
\begin{eqnarray*}
		I_1 
	\nonumber	
	&\le&
		\pr\left(\sup_{s \in \left[0,A_0 \sigma_Y(h(u))\right]} X(s) > u\right)	\\
	&\le&		\label{A1.stat}
		(1 + \varepsilon)\left[1- \E\exp\left(-A_0 \exp(-r + \sqrt{2r} \calN)\right)\right]
\end{eqnarray*}
as $u \toi$.	\\ 
{\it Integral $I_3$}:
Due to \refs{weak.conv}, for sufficiently large $u$,
\[
		I_3	
	\le
		\pr(\calT_u > A_\infty)
	\le
		(1 + \varepsilon)\pr\left(\calT > A_\infty\right).
\]
{\it Integral $I_2$}:
%
\begin{eqnarray}
		I_2 
	\nonumber	
	&=&		
		\int_{A_0}^{A_\infty} \pr\left(\sup_{s \in [0, t\sigma_Y(h(u))]} X(s) > u\right) dF_{\calT_u}(t)	\\
	&\le&		\label{hashorva.conver}
		(1 + \varepsilon) \int_{A_0}^{A_\infty}
	  \left(1 - \E \exp\left(-t \exp\left(-r + \sqrt{2r} \calN \right)\right)\right) dF_{\calT_u}(t)\\
	&\le&		\label{weak.integral}
		(1 + \varepsilon)^2 \int_{A_0}^{A_\infty}
	  \left(1 - \E \exp\left(-t \exp\left(-r + \sqrt{2r} \calN \right)\right)\right) dF_\calT(t),
\end{eqnarray}
where \refs{hashorva.conver} is by Lemma 3.3 in \cite{Tan13} 
and \refs{weak.integral} is due to \refs{weak.conv}, 
and the observation that\\ $1 - \E \exp\left(-t \exp\left(-r + \sqrt{2r} \calN \right)\right)$ is bounded and continuous function 
with respect to $t \in [A_0, A_\infty]$.
Thus, for each $\varepsilon > 0$, $A_\infty > A_0 > 0$,
\begin{eqnarray*}
		\limsup_{u \toi} \pr\left(\sup_{s \in [0, h(u)]} X(Y(s)) > u\right)
	&\le&
		(1 + \varepsilon)^2 \int_{A_0}^{A_\infty}
	  \left(1 - \E \exp\left(-t \exp\left(-r + \sqrt{2r} \calN \right)\right)\right) dF_\calT(t)\\
		&+& (1 + \varepsilon)\left[1- \E\exp\left(-A_0 \exp(-r + \sqrt{2r} \calN)\right)\right]
		+ (1 + \varepsilon)\pr\left(\calT > A_\infty\right).
\end{eqnarray*}
Analogously,
\[
		\liminf_{u \toi} \pr\left(\sup_{s \in [0, h(u)]} X(Y(s)) > u\right)
	\ge
		(1 - \varepsilon)^2 \int_{A_0}^{A_\infty}
	  \left(1 - \E \exp\left(-t \exp\left(-r + \sqrt{2r} \calN \right)\right)\right) dF_\calT(t).
\]
In order to complete the proof it suffices to pass with $A_0 \to 0, A_\infty \toi$, and $\varepsilon \to 0$. 

\Halmos

\section{Acknowledgments}
I would like to thank Krzysztof D\c{e}bicki for a valuable review of the manuscript. 
This work was supported by NCN Grant No 2013/09/D/ST1/03698 (2014--2016).


\end{document}